\documentclass[15pt]{article}

\usepackage{amsmath,amsfonts,amssymb,amsthm,amscd}

\newcounter{stepctr}
{\end{list}}

\newtheorem{thm}{Theorem}[section]
\newtheorem{prop}[thm]{Proposition}
\newtheorem{cor}[thm]{Corollary}

\theoremstyle{definition}
\newtheorem{dfn}[thm]{Definition}
\newtheorem{ex}[thm]{Example}
\newtheorem{exs}[thm]{Examples}
\newtheorem{rema}[thm]{Remark}
\newtheorem{lem}[thm]{Lemma}
\newtheorem{prob*}{Open problem}
\newcommand{\demo}{\begin{proof}}

\newcommand{\h}{\mathcal{H}}

\newcommand{\A}{\mathcal{A}}

\newcommand{\N}{\mathbb{N}}

\newcommand{\C}{\mathbb{C}}

\newcommand{\ind}[1]{{\rm ind}\,({#1})}

\def\ll^2{{\mathcal L}(\ell^2(\N))}
\def\u{{\mathcal U} }

\def\f^0x{{\mathcal F^0}(X) }

\title
{\bf  Property $(z),$ direct sums and a note on an a-Browder type
theorem }

\author{  A. Arroud,  H. Zariouh }

\date{}
\begin{document}

\maketitle \thispagestyle{empty}

\begin{abstract}\noindent\baselineskip=10pt
We characterize the properties $(z)$ and $(az)$  for an operator $T$
whose dual $T^*$ has the SVEP on the complementary of the upper
semi-Weyl spectrum of $T.$ If $S$ and $T$ are Banach space operators
satisfying  property $(z)$ or $(az),$ we give conditions on $S$ and
$T$ to ensure the preservation of these properties  by
the direct sum $S\oplus T.$ Some results are given for multipliers
and in general for $(H)$-operators. Also we give a correct proof of
\cite[Theorem 2.3]{SZ} which was proved by using the equality
$\sigma_p^0(S\oplus T)= \sigma_p^0(S)\cup \sigma_p^0(T).$ However
this equality is not true; we give  counterexamples to show that.
\end{abstract}

 \baselineskip=15pt
  \baselineskip=15pt
\section{Introduction and preliminaries}
Let $X$ denote an infinite dimensional complex Banach space, and
denote by $L(X)$ the algebra of all bounded linear operators on $X.$
For $T\in L(X),$ we  denote by $\alpha(T)$ the dimension of the
kernel $N(T)$ and by $\beta(T)$ the codimension of the range $R(T).$
By $\sigma (T), \sigma_a(T), \sigma_s(T),$ we denote the spectrum,
 the approximate spectrum and the surjectivity
spectrum of $T,$
respectively.\\
 Recall that $T$ is said to be \textit{upper semi-Fredholm}, if $R(T)$ is
closed and $\alpha(T) <\infty,$ while $T$ is called \textit{lower
semi-Fredholm}, if $R(T)$ is closed and  $\beta(T) < \infty.$ $T\in
L(X)$ is said to be \textit{semi-Fredholm} if $T$ is either an upper
semi-Fredholm or a lower semi-Fredholm operator. $T$ is
\textit{Fredholm} if $T$ is upper semi-Fredholm and lower
semi-Fredholm. If $T$ is semi-Fredholm then the index of $T$ is
defined by $\ind T = \alpha(T) -\beta(T).$ For an operator $T \in
L(X),$ the \textit{ascent} $a(T)$ and the \textit{descent} $d(T)$
are defined by $a(T) = \inf\{n\in\N : N(T^n) = N(T^{n+1})\}$ and
$d(T)= \inf\{n\in\N :R(T^n) = R(T^{n+1})\},$ respectively; the
infimum over the empty set is taken $\infty.$ If the ascent and the
descent of $T$  are both finite, then $a(T) = d(T)=p,$  and $R(T
^p)$ is closed. An operator $T$ is said to be \textit{Weyl} if it is
Fredholm of index zero. It is called \textit{upper semi-Weyl}
(resp., \textit{lower semi-Weyl}) if it is upper semi-Fredholm of
index $\leq 0$ (resp., lower semi-Fredholm of index $\geq 0$). $T$
is called \textit{upper semi-Browder} if it is an upper
semi-Fredholm operator with finite ascent and it is called
\textit{Browder} if it is Fredholm of finite
ascent and descent. \\
If $T\in L(X)$ and $ n \in \N,$ we denote  by $T_n$ the restriction
of  $T$ on $R(T^n).$ $T$ is said to be \textit{upper semi b-Weyl}, if there
exists $n\in\N$ such that $R(T^n) $ is closed and $T_n:
R(T^n)\rightarrow R(T^n)$ is upper semi-Weyl.\\
 We
recall that a complex number $\lambda\in\sigma (T)$ is a
\textit{pole} of the resolvent of $T,$ if $T-\lambda I$ has finite
ascent and finite descent and $\lambda\in\sigma_a(T)$ is a
\textit{left pole} of $T$ if $p=a(T-\lambda I) <\infty$ and
$R(T^{p+1})$ is closed.

In the following  list, we summarize  the notations and symbols
needed later.

\smallskip

\noindent $\mbox{iso}\,A$: isolated points of a  subset $A\subset \mathbb{C},$\\
 \noindent $\mbox{acc}\,A$: accumulations  points of a subset $A\subset \mathbb{C},$\\
 \noindent $A^C$: the complementary of a subset $A\subset \mathbb{C},$\\
\noindent $D(0, 1)$: the closed unit disc in $\mathbb{C},$\\
$C(0, 1)$: the  unit circle of $\mathbb{C},$\\
 $\h(\sigma(T))$: the set of all analytic functions defined on an open
neighborhood of  $\sigma(T),$\\
$p_0(T)$: poles of $T,$\\
$p_{00}(T)$: poles of $T$ of finite rank,\\
$p_0^a(T)$: left  poles of $T,$\\
$p_{00}^a(T)$: left  poles of $T$ of finite rank,\\
$\sigma_{p}(T)$:  eigenvalues of $T,$\\
$\sigma_{p}^0(T)$: eigenvalues of $T$ of finite multiplicity,\\
$\pi_0(T):=\mbox{iso}\,\sigma(T)\cap\sigma_{p}(T),$ \\
$\pi_{00}(T):=\mbox{iso}\,\sigma(T)\cap\sigma_{p}^0(T),$\\
$\pi_0^a(T):=\mbox{iso}\,\sigma_a(T)\cap\sigma_{p}(T),$\\
$\pi_{00}^a(T):=\mbox{iso}\,\sigma_a(T)\cap\sigma_{p}^0(T),$\\
$\sigma_{uf}(T) = \{\lambda\in \C : T-\lambda I$  is not upper
semi-Fredholm\}: upper semi-Fredholm spectrum,\\
$\sigma_{lf}(T) = \{\lambda\in \C : T-\lambda I$  is not lower
semi-Fredholm\}: lower semi-Fredholm spectrum,\\
$\rho(T)=\C\setminus\sigma(T);$ $\rho_a(T)=\C\setminus\sigma_a(T);$ $\rho_{uf}(T)=\C\setminus\sigma_{uf}(T),$\\
$\sigma_{b}(T)=\sigma(T)\setminus p_{00}(T)$:  Browder spectrum of $T,$\\
$\sigma_{ub}(T)=\sigma_a(T)\setminus p_{00}^a(T)$:  upper Browder spectrum of $T,$\\
$\sigma_{w}(T)$: Weyl spectrum of $T,$\\
$\sigma_{uw}(T)$: upper semi-Weyl spectrum of $T,$\\
$\sigma_{lw}(T)$: lower semi-Weyl spectrum of $T,$\\
$\sigma_{ubw}(T)$: upper semi-b-Weyl spectrum of $T,$

\begin{dfn}\cite{barnes}, \cite{rako}, \cite{Z1} \label{dfn1}Let $T\in L(X).$ $T$ is said to satisfy\\
a) a-Browder's theorem if  $\sigma_a(T)\setminus\sigma_{uw}(T)=p_{00}^a(T).$\\
b) a-Weyl's theorem if  $\sigma_a(T)\setminus\sigma_{uw}(T)=\pi_{00}^a(T).$\\
c) property $(z)$  if $\,\sigma(T)\setminus\sigma_{uw}(T)=\pi_{00}^a(T).$\\
d)  property $(az)$ if
$\sigma(T)\setminus\sigma_{uw}(T)=p_{00}^a(T).$\\
e) property $(gaz)$ if
$\sigma(T)\setminus\sigma_{ubw}(T)=p_{0}^a(T).$\\
f)  property $(gz)$
if $\sigma(T)\setminus\sigma_{ubw}(T)=\pi_0^a(T).$

\end{dfn}

The relationship between a-Browder's theorem and the properties
given in the precedent definition was studied in \cite{Z1}, and it
is summarized as follows:
\begin{center} property $(gz)\,\,\Longrightarrow \,\, \mbox{ property } (z)\,\,\Longrightarrow\,\, \mbox{ property }(az)\,\,\Longrightarrow \,\,\mbox{ a-Browder's theorem
}$          \end{center}
          \begin{center}
    $\mbox{ property }(az)\,\,\Longleftrightarrow \,\,\mbox{ property } (gaz) $
      \end{center}
Moreover, in \cite{Z1}  counterexamples were given to  show that the
reverse of
 each implication in the diagram is not true.

The following property named SVEP has relevant role in local
spectral theory. For more details see the recent monographs
\cite{Aie} and \cite{LN}.

\begin{dfn}\cite{LN} An operator $T\in L(X)$ is said to have the \textit{single valued
extension property} (SVEP ) at $\lambda_{0}\in\mathbb{C}$ , if for
every open neighborhood $\u$ of $\lambda_{0}$, the only analytic
function $f:\u\longrightarrow X$ which satisfies the equation
$(T-\lambda I)f(\lambda)=0$ for all $\lambda\in\u$ is the function
$f\equiv0$. An operator $T\in L(X)$ is said to have the SVEP if $T$
has  the SVEP at every point $\lambda\in\mathbb{C}$.\end{dfn}

It follows easily  that $T\in L(X)$ has the  SVEP at every point of the boundary $\partial\sigma(T)$ of the spectrum $\sigma(T).$ In particular, $T$ has the SVEP at every    point of
$\mbox{iso}\,\sigma(T).$ We also have \[a(T-\lambda_0 I)<\infty \Longrightarrow
\mbox{ T has the SVEP at } \lambda_0, \,\,\,(I_1)\] and dually
\[d(T-\lambda_0 I)<\infty \Longrightarrow \mbox{ $T^*$ has the SVEP
at } \lambda_0,\,\,\,(I_2)\] where $T^*$ denotes the dual of $T$,
see \cite[Theorem 3.8]{Aie}. Furthermore, if $T-\lambda_0 I$ is
semi-Fredholm then the implications above are equivalences.

\section{Properties $(az),\,(z)$ and SVEP  }

An important class of operators is given by the multipliers on a
semi-simple Banach algebra $\A.$ Recall that an operator $T\in
L(\A)$ is a multiplier if $aT(b) = T(a)b,\,\,\,\forall\,\, a, b \in
\A.$

\begin{prop}\label{pr0} If $T$ is a multiplier on a semi-simple Banach algebra $\A,$ then\\
    i) $T$ has the SVEP, $a(T)\leq 1$ and $\alpha(T)\leq \beta(T).$\\
    ii) If in addition $\A$ is commutative regular and Tauberian then properties $(z)$ and $(az)$ hold for $T$
    and they hold for $T^*$ too if $T^*$ has the SVEP.
\end{prop}
\begin{proof}i) See \cite[Theorem 4.32]{Aie} and \cite[Theorem
3.4]{Aie}.\\
ii) From \cite[Corollary 5.88]{Aie} we have
$\sigma_{a}(T)=\sigma(T)=\sigma_{a}(T^*)$ and from \cite[Theorem
5.118]{Aie}, a-Weyl's theorem holds for $T$ and it holds for $T^*$ if
it has the SVEP. The conclusion follows from \cite[Theorem 2.4]{Z1}
and since property $(z)$ entails property $(az).$
\end{proof}

Now we give a characterization of property $(az)$ for a linear
operator $T$ whose dual $T^*$ has the SVEP on ${\sigma_{uw}(T)}^C.$
By duality we give a similar result for $T^*.$

\begin{thm} \label{prop1}Let $T\in L(X),$ then:\\
 i) $T^*$ has the SVEP on  ${\sigma_{uw}(T)}^C$ if and only if $T$
 satisfies property $(az).$\\
 ii) $T$ has the SVEP on  ${\sigma_{lw}(T)}^C$ if and only if $T^*$
 satisfies property $(az).$
\end{thm}

\begin{proof}
$T^*$  has the SVEP on ${\sigma_{uw}(T)}^C$ then from \cite[Theorem
2.2]{AiP}, $T$ satisfies a-Browder's theorem
$\sigma_a(T)\setminus\sigma_{uw}(T)=p_{00}^a(T).$ We only have to
prove that $\sigma(T)=\sigma_a(T).$ Let $\mu_0\not\in\sigma_{a}(T)$
be arbitrary, then $\mu_0\not\in\sigma_{uw}(T),$  $T-\mu_0 I$ is
injective and $R(T)$ is closed. So $T-\mu_0 I$ is an upper
semi-Fredholm operator, and by the implication $(I_2)$ above we
conclude that   $\beta(T-\mu_0 I)=0.$ Hence
$T-\mu_0 I$ is surjective and $\mu_0\not\in\sigma(T).$\\ Conversely,
suppose that $T$ satisfies $(az).$ Let
$\lambda_0\in\sigma_{uw}(T)^C$ be arbitrary. We distinguish two
cases: if $\lambda_0\not\in \sigma(T)= \sigma(T^*)$ then $T^*$ has
the SVEP at $\lambda_0.$ If $\lambda_0\in \sigma(T)$ then
$\lambda_0\in p_{00}^a(T).$ Thus $\lambda_0$ is isolated in
$\sigma_a(T)=\sigma_{s}(T^*)$ and hence $T^*$ has the SVEP at
$\lambda_0.$

 ii) By the duality between $T$ and $T^*$ the proof goes similarly
 with (i).
\end{proof}

\begin{rema}\label{rema0}
In   Theorem \ref{prop1}, we cannot replace the SVEP for $T^*$ on ${\sigma_{uw}(T)}^C$ (resp., the SVEP for $T$ on ${\sigma_{lw}(T)}^C$) by the SVEP for $T$ on ${\sigma_{uw}(T)}^C$ (resp., the SVEP for $T^*$ on ${\sigma_{lw}(T)}^C).$
 Here  and elsewhere $R$ and $L$ denote the unilateral right and left shifts operators on $\ell^2(\mathbb{\N})$ defined  by
 $R(x_1,x_2,\ldots)=(0,x_1,x_2, x_3,\ldots)$ and $L(x_1,x_2, x_3,\ldots)=(x_2, x_3, \ldots).$  Evidently $L^*=R$ has the SVEP, but $R$  does not satisfy property $(az),$ since $\sigma(R)=D(0, 1),$ $\sigma_{uw}(R)=C(0, 1)$ and $p_{00}^a(R)=\emptyset.$\\
 Another example on $\ell^2(\mathbb{\N})$ is given  by $T(x_1,x_2, x_3,\ldots)=(\frac{1}{2}x_1, 0, x_2,  x_3, \ldots).$ We have $T^*$ has the SVEP; since its approximate  spectrum $\sigma_a(T^*)=C(0, 1)\cup\{\frac{1}{2}\}$ has empty interior. But, $T^*$ does not satisfy property $(az);$  since $\sigma(T^*)=D(0, 1),$  $\sigma_{uw}(T^*)=C(0, 1)$ and $p_{00}^a(T^*)=\{\frac{1}{2}\}.$  \end{rema}

\begin{prop}Suppose that the dual $T^*$ of  $T\in L(X)$ has the SVEP, then \\
 i) If  $Q\in L(X)$ is  a quasi-nilpotent operator which commutes  with $T,$ then
  $f(T)+Q$ and $f(T+Q)$ satisfy property $(az)$  for every $f\in\h(\sigma(T)).$ \\
ii) If  $K\in L(X)$ is  an algebraic (resp., $F$ a finite rank)
operator which commutes  with $T,$ then $T+K$ (resp., $T+F$)
satisfies property $(az).$
\end{prop}

\begin{proof}  i) We know from \cite[Theorem 2.40]{Aie} that if   $T^*$ has the SVEP, then $(f(T))^*=f(T^*)$ has the SVEP. Since $Q$ is quasi-nilpotent and commutes with $T$ then from \cite[Corollary 2.12]{Aie}, $T^*+Q^*$ has the SVEP. It follows that
$(f(T+Q))^*=f((T+Q)^*)$ and   $(f(T)+Q)^*$ have the SVEP. From Theorem \ref{prop1},  $f(T)+Q$ and $f(T+Q)$ satisfy property $(az).$ \\
ii) If $K$ is algebraic and commutes with $T,$ then $K^*$ is also
algebraic and commutes with $T^*.$ From \cite[Theorem 2.14]{Ai3},
$T^*+K^*=(T+K)^*$ has the SVEP. Hence $T+K$ satisfies property
$(az).$ If $F$ is a finite rank operator and commutes with $T,$ then
$T^*+F^*=(T+F)^*$ has the SVEP, see the proof of  \cite[Lemma
2.8]{Ai4}. Hence $T+F$ satisfies property $(az).$
\end{proof}

From Theorem \ref{prop1} and \cite[Theorem 3.6]{Z1} and
\cite[Corollary 3.7]{Z1}, we obtain immediately the following
characterizations for properties $(z)$ and $(gz).$ We recall that
$T\in L(X)$ is said to be \textit{finitely a-polaroid}  if every
isolated point of $\sigma_a(T)$ is a left pole  of $T$ of finite
rank and is said to be
 \textit{a-polaroid} if every isolated point of $\sigma_a(T)$ is a left pole  of $T.$ Note that every finitely
  a-polaroid operator is a-polaroid, but the converse is not true.
  For this, consider the operator $P$ defined on $\ell^2(\mathbb{\N})$
  by: $P(x_1,x_2,\ldots)=(0,x_2, x_3,\ldots).$ Then
  $\mbox{iso}\,\sigma_a(T)=\{0, 1\}=p_0^a(P)$ and $p_{00}^a(P)=\{0\}.$

\begin{cor} \label{cor1}Let $T\in L(X),$ then:\\
i)  $T$ satisfies property $(z)$ if and only if  $T^*$ has the SVEP
on ${\sigma_{uw}(T)}^C$ and $\pi_{00}^a(T)=p_{00}^a(T).$   In particular, if $T$ is finitely  a-polaroid, then
$T$ satisfies property $(z)$ if and only if  $T^*$ has the SVEP
on ${\sigma_{uw}(T)}^C.$\\
ii)  $T$ satisfies property $(gz)$ if and only if  $T^*$ has the SVEP
on ${\sigma_{uw}(T)}^C$ and $\pi_{0}^a(T)=p_{0}^a(T).$   In particular, if $T$ is   a-polaroid, then
$T$ satisfies property $(gz)$ if and only if  $T^*$ has the SVEP
on ${\sigma_{uw}(T)}^C.$\\
iii)  $T^*$ satisfies property $(z)$ if and only if  $T$ has the
SVEP on ${\sigma_{lw}(T)}^C$ and $\pi_{00}^a(T^*)=p_{00}^a(T^*).$\\
iv)  $T^*$ satisfies property $(gz)$ if and only if  $T$ has the
SVEP on ${\sigma_{lw}(T)}^C$ and $\pi_{0}^a(T^*)=p_{0}^a(T^*).$
\end{cor}

\begin{rema} In  Corollary \ref{cor1}, we do not expect neither property $(z)$ nor property $(gz)$ for
an operator $T$ with only (as hypothesis) the SVEP of its dual on
${\sigma_{uw}(T)}^C.$ Indeed, let $T$ be the operator defined on
$\ell^2(\mathbb{\N})$ by $T(x_1,x_2,\ldots)=(\frac{x_2}{2},
\frac{x_3}{3},\ldots)$ then $T^*$ has the SVEP, but $T$ does not
satisfy property $(z)$ and then it does not satisfy property $(gz)$ too. Note that here $\pi_{00}^a(T)=\pi_{0}^a(T)=\{0\}$ and
$p_{00}^a(T)=p_{0}^a(T)=\emptyset.$

\end{rema}

\section{Properties $(az),$ $(z)$ and direct sums}

In the following,
$Y$ denotes an infinite dimensional  complex Banach space.

\begin{dfn} Let $T\in L(X)$ and $S\in L(Y)$. We say that $T$ and $S$ have a shared stable sign
index if for each $\lambda\in \rho_{uf}(T)$ and each
$\mu\in\rho_{uf}(S)$, $\mbox{ind}(T-\lambda I)$ and
$\mbox{ind}(S-\mu I)$ have the same sign.
\end{dfn}

\begin{exs}
We give some examples of operators with shared stable sign index.

    (a) For an hyponormal operator $T$ on a Hilbert space we always have $
    \mbox{ind}(T-\lambda I)\leq 0,$ for each $\lambda\in \rho_{uf}(T).$ Hence two hyponormal operators acting on Hilbert
  spaces have a shared stable sign index.

     (b) By Proposition \ref{pr0}, two multipliers on semi-simple Banach
     algebras have a shared stable sign index. Moreover, according
     to \cite[Theorem
4.5]{AiL}, any multiplier $T$ on a commutative semi-simple algebra
is Fredholm if and only it is upper semi-Fredholm and in this case
$\mbox{ind} (T)=0.$

     c) If $S$ and $T$ are two operators having the SVEP on the complementary of their
      upper semi-Fredholm spectra respectively, then they
     have a shared stable sign index. The same occurs when $S^*$ and $T^*$ have the  SVEP on the complementary of their
      lower semi-Fredholm spectra respectively.

Suppose for instance that $S$ has the SVEP on $\rho_{uf}(S)$  and
$T$ has the SVEP on $\rho_{uf}(T).$ If $\lambda\in \rho_{uf}(S)$ and
$\mu\in\rho_{uf}(T)$ then $S-\lambda I$ and $T-\mu I$ are upper
semi-Fredholm and since $S$ and $T$ have the SVEP at $\lambda$ and
$\mu$ respectively, then from the implication $(I_1)$ above we have
$a(S-\lambda I)$ and $a(T-\mu I)$ are finite. Hence from
\cite[Theorem 3.4]{Aie}, $S$ and $T$  have a shared stable index.

Note that the definition of \textsl{shared stable sign index} used
here is weaker and  slightly different from \cite[Definition
1.2]{Az}.
\end{exs}

\begin{lem}\label{lem1} Let $S\in L(X)$ and $T\in
L(Y),$ then the following properties hold:\\
i) $\sigma_{uw}(S\oplus
T)\subseteq\sigma_{uw}(S)\cup\sigma_{uw}(T).$\\
ii) If $S\oplus T$ satisfies a-Browder's theorem, then
$\sigma_{uw}(S\oplus T)=\sigma_{uw}(S)\cup
\sigma_{uw}(T).$ \\
 iii) If $S$ and $T$ have a shared
stable sign index, then $\sigma_{uw}(S\oplus
T)=\sigma_{uw}(S)\cup\sigma_{uw}(T).$\\ In particular, this equality
holds if  $S$ and $T$ have the SVEP .

\end{lem}

\begin{proof} i) If $\lambda\not\in\sigma_{uw}(S)\cup\sigma_{uw}(T)$ , then $S-\lambda I$ and $T-\lambda I$
are upper semi-Weyl operators. Hence $(S\oplus T)-\lambda I$ is an
upper semi-Fredholm operator with $\mbox{ind}((S\oplus T)-\lambda
I)=\mbox{ind}(S-\lambda I)+\mbox{ind}(T-\lambda I)\leq0.$ So
$\lambda\not\in\sigma_{uw}(S\oplus T).$

ii) If $S\oplus T$ satisfies a-Browder's theorem, then
$\sigma_{uw}(S\oplus T)=\sigma_{ub}(S\oplus T).$\\ As
 $\sigma_{ub}(S\oplus T)=\sigma_{ub}(S)\cup\sigma_{ub}(T),$ then  $\sigma_{uw}(S\oplus T)=\sigma_{ub}(S)\cup\sigma_{ub}(T).$ Since the
 inclusion\\
  $\sigma_{uw}(S)\cup\sigma_{uw}(T)\subset  \sigma_{ub}(S)\cup\sigma_{ub}(T)$ is always true, we then have
  $\sigma_{uw}(S)\cup\sigma_{uw}(T)\subset \sigma_{uw}(S\oplus T).$ We conclude by i) that $\sigma_{uw}(S\oplus
T)=\sigma_{uw}(S)\cup\sigma_{uw}(T).$

 iii) If  $\lambda\not\in\sigma_{uw}(S\oplus
T),$ then  $(S\oplus T)-\lambda I$ is an upper semi-Weyl operator.
It follows that both $S-\lambda I$ and $T-\lambda I$ are upper
semi-Fredholm . Since $\mbox{ind}((S\oplus T)-\lambda
I)=\mbox{ind}(S-\lambda I)+\mbox{ind}(T-\lambda I)\leq0$ and $S$ and
$T$ have a shared stable sign index, we have $\mbox{ind}(S-\lambda
I)\leq0$ and $\mbox{ind}(T-\lambda I)\leq0.$ Hence $\lambda\not\in
\sigma_{uw}(S)\cup\sigma_{uw}(T).$
\end{proof}

\begin{rema}\label{rema1}  The inclusion showed in the  first  statement  is proper: for this let $R$ and $L$ be the operators defined on $\ell^2(\mathbb{\N})$  in Remark \ref{rema0}.
 We then  have $\sigma_{uw}(R\oplus L)=C(0, 1)$ and $\sigma_{uw}(R)\cup \sigma_{uw}(L)=D(0, 1).$ Observe that
  $L$ and $R$ are upper semi-Fredholm operators with $\mbox{ind}(R)=-1$ and $\mbox{ind}(L)=1.$
\end{rema}

In the next theorem,  we characterize the stability of
property
 $(az)$ under diagonal operator matrices in terms of upper semi-Weyl  spectra of its components.

\begin{thm}\label{T1}Suppose that  $S\in L(X)$ and $T\in L(Y)$ satisfy
property $(az),$ then
 $S\oplus T$ satisfies property $(az)$ if and only if
 $\sigma_{uw}(S\oplus T)=\sigma_{uw}(S)\cup
\sigma_{uw}(T).$\end{thm}

\begin{proof}
($\Leftarrow$) Since $S$ and $T$ satisfy property $(az)$ then by
\cite[Theorem 3.2]{Z1}, $S$ and $T$ satisfy a-Browder theorem,
$\sigma(S)= \sigma_a(S)$ and $\sigma(T)=\sigma_a(T).$  It follows
that $\sigma_{uw}(S)=\sigma_{ub}(S),$
$\sigma_{uw}(T)=\sigma_{ub}(T)$ and $\sigma(S\oplus T)=
\sigma_a(S\oplus T).$ Moreover as $\sigma_{ub}(S\oplus
T)=\sigma_{ub}(S)\cup\sigma_{ub}(T)$ we have : $\sigma_{ub}(S\oplus
T)= \sigma_{uw}(S)\cup\sigma_{uw}(T)= \sigma_{uw}(S\oplus T).$ This
implies that $S\oplus T$ satisfies a-Browder's theorem. According
to \cite[Theorem 3.2]{Z1} we deduce that $S\oplus T$ satisfies
property $(az).$

($\Rightarrow$) Suppose that property $(az)$ holds for $S\oplus T$
then  $S\oplus T$ satisfies a-Browder's theorem, and by Lemma
\ref{lem1}, it follows that $\sigma_{uw}(S\oplus
T)=\sigma_{uw}(S)\cup \sigma_{uw}(T).$
\end{proof}

 As a consequence of Lemma \ref{lem1} and Theorem \ref{T1} we have the next
 corollary.
\begin{cor}\label{thm1} Suppose that  $S\in L(X)$ and $T\in L(Y)$ satisfy
property $(az),$ and have a shared stable sign
index then
 $S\oplus T$ satisfies property $(az).$\end{cor}

We recall that $\sigma_{p}(S\oplus
T)=\sigma_{p}(S)\cup\sigma_{p}(T)$ and $\alpha(S\oplus
T)=\alpha(S)+\alpha(T)$
 for every pair of operators so that $\sigma_p^0(S\oplus T)=\{\lambda\in
\sigma_p^0(S)\cup\sigma_p^0(T) \,|\, \alpha(S-\lambda
I)+\alpha(T-\lambda I)<\infty\}.$ Moreover, if  $A$  and $B$ are
bounded subsets of the complex plane $\mathbb{\C}$ then
$\mbox{acc}(A \cup B)=\mbox{acc}(A)\cup\mbox{acc}(B).$

\begin{lem}\label{lem2} If  $S\in L(X)$ and $T\in L(Y)$ satisfy $\sigma_{p}^0(S)=\sigma_{p}^0(T),$ then

i) $\pi_{00}^a(S\oplus T)=\pi_{00}^a(S)\cap\pi_{00}^a(T),$

 ii) $\pi_{00}(S\oplus T)=\pi_{00}(S)\cap\pi_{00}(T),$

  iii) $p_{00}(S\oplus T)=p_{00}(S)\cap p_{00}(T),$

 iv) $p_{00}^a(S\oplus T)=p_{00}^a(S)\cap p_{00}^a(T).$

 And if $\sigma_{p}(S)=\sigma_{p}(T),$ then

v) $\pi_{0}^a(S\oplus T)=\pi_{0}^a(S)\cap\pi_{0}^a(T).$
\end{lem}
\begin{proof} i) As   $\sigma_p^0(T)=\sigma_p^0(S)$\\
 then   $\sigma_p^0(S\oplus
T)=\sigma_p^0(S)=\sigma_p^0(T)$ and so
$\pi_{00}^a(S)\cap\rho_a(T)=\pi_{00}^a(T)\cap\rho_a(S)=\emptyset.$
Thus we have:
\begin{eqnarray*}\pi_{00}^a(S\oplus T)
&=&\{\mbox{iso\,}\sigma_a(S\oplus
T)\}\cap\sigma_p^0(S\oplus T)\\
&=&\{\mbox{iso}[\sigma_a(S)\cup
\sigma_a(T)]\}\cap\sigma_p^0(S)\\
&=&\{[\sigma_a(S)\cup\sigma_a(T)]\setminus\mbox{acc}[\sigma_a(S)\cup
\sigma_a(T)]\}\cap\sigma_p^0(S)\nonumber\\
&=&\{[\sigma_a(S)\cup\sigma_a(T)]\setminus[\mbox{acc\,}\sigma_a(S)\cup
\mbox{acc\,}\sigma_a(T)]\}\cap\sigma_p^0(S)\nonumber\\
&=&\{[\mbox{iso\,}\sigma_a(S)\cap\rho_a(T)]\cup[\mbox{iso\,}\sigma_a(T)\cap\rho_a(S)]
\cup[\mbox{iso\,}\sigma_a(S)\cap\mbox{iso\,}\sigma_a(T)]\}\cap\sigma_p^0(S)\nonumber\\
&=& [\pi_{00}^a(S)\cap\rho_a(T)]\cup[\pi_{00}^a(T)\cap\rho_a(S)]\cup
[\pi_{00}^a(S)\cap \pi_{00}^a(T)] \\
&=& \pi_{00}^a(S)\cap \pi_{00}^a(T) \label{eq8}.\end{eqnarray*}

The proof of ii)  goes similarly with i).

 iii) Since  $p_{00}(S)\cap\rho(T)= p_{00}(T)\cap\rho(S)=\emptyset$ then we have
\begin{eqnarray*}p_{00}(S\oplus T)
&=&\sigma (S\oplus T)\setminus
 \sigma_b(S\oplus T)\\
&=&[\sigma(S)\cup \sigma(T)]\setminus [\sigma_b(S)\cup
\sigma_b(T)]\\
&=&\{[\sigma(S)\setminus\sigma_b(S)]\cap\rho(T)\}\cup\{[\sigma(T)\setminus\sigma_b(T)]\cap\rho(S)\} \cup \{[\sigma(S)\setminus\sigma_b(S)]\cap[\sigma(T)\setminus\sigma_b(T)]\}\nonumber\\
&=&[p_{00}(S)\cap\rho(T)]\cup[p_{00}(T)\cap\rho(S)]\cup[p_{00}(S)\cap
 p_{00}(T)]\\
 &=&p_{00}(S)\cap
 p_{00}(T).\end{eqnarray*}

 The proof of iv)  goes similarly with iii) and the proof of v)  goes similarly with
 i).
\end{proof}

\begin{ex}\label{ex2}
Generally the equalities in Lemma \ref{lem2} are not true and the
hypothesis assumed on the point spectra are essential as we can see
in the following examples:

 a) Let $T\in L(\mathbb{\C}^n)$ be a non trivial
nilpotent operator and consider $R\in L(\ell^2(\mathbb{\N}))$ the
unilateral right shift. We have $\pi_{00}^a(R\oplus
T)=\pi_{0}^a(R\oplus T)=\{0\},$ but $\pi_{00}^a(R)\cap
\pi_{00}^a(T)=\pi_{0}^a(R)\cap \pi_{0}^a(T)=\emptyset.$ Here
$\{0\}=\sigma_p(T)=\sigma_p^0(T)\neq
\sigma_p(R)=\sigma_p^0(R)=\emptyset.$ On the other hand, if we
consider the operator $A$ defined on $L(\ell^2(\mathbb{\N}))$ by
$A(x_1,x_2, x_3,\ldots)=(0,\frac{x_1}{2}, \frac{x_2}{3},
\frac{x_3}{4},\ldots),$ then it is easily seen that
$\pi_{00}(A\oplus T)=\{0\},$ but $\pi_{00}(A)\cap
\pi_{00}(T)=\emptyset.$ Note also that
$\sigma_p^0(T)\neq\sigma_p^0(A)=\emptyset.$

 b) Consider the operator
$K$ defined on $\ell^2(\mathbb{\N})$ by $ K(x_1,x_2,
x_3,\ldots)=(x_1,\frac{x_2}{2}, \frac{x_3}{3},
\frac{x_3}{4},\ldots),$ and $T\in L(\mathbb{\C}^n)$ a non trivial
nilpotent operator. Then $\sigma_p^0(K)=p_{00}(K)= p_{00}^a(K)=
\{\frac{1}{n}\,|\,n\in\N^*\};$ $\sigma_p^0(T)=p_{00}(T)=
p_{00}^a(T)=\{0\}$ and $p_{00}(K\oplus T)=p_{00}^a(K\oplus
T)=\{\frac{1}{n}\,|\,n\in\N^*\},$ but  $p_{00}(K)\cap p_{00}(T=
p_{00}^a(K)\cap
p_{00}^a(T)=\emptyset.$

\end{ex}

In the next, we give conditions to ensure the transmission  of property $(z)$ from  $S\in L(X)$ and $T\in L(Y)$ to their direct sum  $S\oplus T.$

\begin{thm}\label{T3} Suppose  $S\in L(X)$ and $T\in L(Y)$ satisfy $\sigma_{p}^0(S)=\sigma_{p}^0(T).$ If property $(z)$ holds for $S$
and $T,$ then it holds for
 $S\oplus T$ if and only if  $\sigma_{uw}(S\oplus T)=\sigma_{uw}(S)\cup
\sigma_{uw}(T).$\end{thm}

\begin{proof} ($\Rightarrow$)  If $S\oplus T$ satisfies property $(z),$ then
 by \cite[Theorem 3.6]{Z1}, it satisfies property
 $(az).$ As seen in the proof of Theorem \ref{T1}, we conclude that    $\sigma_{uw}(S\oplus
T)=\sigma_{uw}(S)\cup \sigma_{uw}(T).$

($\Leftarrow$)
Since $S$ and $T$  satisfy property $(z),$
 then
\begin{eqnarray*} \sigma(S\oplus
T)\setminus\sigma_{uw}(S\oplus
T)&=&[\sigma(S)\cup\sigma(T)]\setminus[\sigma_{uw}(S)\cup\sigma_{uw}(T)]\\
&=&[(\sigma(S)\setminus\sigma_{uw}(S))\cap\rho(T)]\cup[(\sigma(T)\setminus\sigma_{uw}(T))\cap\rho(S)]\nonumber\\
&&\cup[(\sigma(S)\setminus\sigma_{uw}(S))\cap (\sigma(T)\setminus\sigma_{uw}(T))]\nonumber\\
&=& [\pi_{00}^a(S)\cap\rho(T)]\cup[\pi_{00}^a(T)\cap\rho(S)]\cup[\pi_{00}^a(S)\cap
\pi_{00}^a(T)].\label{eq7}\end{eqnarray*} As   $\pi_{00}^a(T)\cap\rho(S)
=\pi_{00}^a(S)\cap\rho(T)=\emptyset,$ then
 $\sigma(S\oplus T)\setminus\sigma_{uw}(S\oplus T)=\pi_{00}^a(S)\cap \pi_{00}^a(T).$  We conclude by Lemma
 \ref{lem2} that property $(z)$ holds for $S\oplus T.$\end{proof}

We recall that $T\in L(X)$ is said to be \textit{a-isoloid} if every
isolated point of $\sigma_a(T)$ is an  eigenvalue of $T,$ and is
said to be \textit{isoloid} if every isolated point of $\sigma(T)$
is an  eigenvalue of $T.$ Clearly, if $T$ is  a-isoloid then it is
isoloid. However the converse is not true. Consider the following
example: let $T=R\oplus Q$ the operator on
$\ell^2(\mathbb{\N})\oplus\ell^2(\mathbb{\N}),$ where $R$ is the
unilateral right shift  and $Q$ is an injective quasi-nilpotent
operator. Then $\sigma(T)=D(0, 1)$ and $\sigma_a(T)=C(0,
1)\cup\{0\}.$ Therefore $T$ is isoloid but not a-isoloid.

\begin{lem}\label{cl} If   $S\in L(X)$ and $T\in L(Y)$
 are a-isoloid then $$\pi_{00}^a(S\oplus T)=[\pi_{00}^a(S)\cap\rho_a(T)]\cup[\pi_{00}^a(T)\cap\rho_a(S)]\cup
[\pi_{00}^a(S)\cap \pi_{00}^a(T)]$$
\end{lem}

\begin{proof} It is easy to see that the inclusion $\supset$ is
always true without condition on $S$ and $T.$\\
Suppose now that $\lambda\in \pi_{00}^a(S\oplus T),$ then $\lambda$
is isolated in $\sigma_a(S\oplus T)=\sigma_a(S)\cup\sigma_a(T).$\\
\textsl{Case}1: $\lambda\in\sigma_a(S)\setminus\sigma_a(T).$ As
$\sigma_p^0(S\oplus T)\subset \sigma_p^0(S)\cup \sigma_p^0(T)$ is
always true then $\lambda\in \pi_{00}^a(S)\cap\rho_a(T).$\\
\textsl{Case}2: $\lambda\in\sigma_a(T)\setminus\sigma_a(S).$
Similarly with case1 we conclude that $\lambda\in
\pi_{00}^a(T)\cap\rho_a(S).$\\
 \textsl{Case}3:  $\lambda\in\sigma_a(T)\cap\sigma_a(S).$ Then
 $\lambda\in \mbox{iso}\,\sigma_a(S)\cap \mbox{iso}\,\sigma_a(T)$
 and since $S$ and $T$ are a-isoloid and $\lambda$ is an eigenvalue of
 finite multiplicity of $S\oplus T,$ then $\lambda\in \pi_{00}^a(S)\cap \pi_{00}^a(T).$
\end{proof}

In the next theorem, we give a similar characterization of the
property $(z)$ for $S\oplus T$ under the hypothesis that $S$ and $T$
are a-isoloid. Notice that the condition
``$\sigma_{p}^0(S)=\sigma_{p}^0(T)$'' assumed in Theorem \ref{T3},
and the condition  ``$S$ and $T$ being a-isoloid'' of Theorem
\ref{T4} below  are independent: indeed, the operators $T$ and $R$
defined in Example \ref{ex2} are a-isoloid, but
$\sigma_p^0(R)=\emptyset$ and $\sigma_p^0(T)=\{0\}.$ Conversely, if
we consider the operator $A$ defined on $\ell^2(\mathbb{\N})$ by $
A(x_1,x_2, x_3,\ldots)=(0,\frac{x_1}{2}, \frac{x_2}{3},
\frac{x_3}{4},\ldots),$ and  the Volterra operator on the Banach
space $C[0, 1]$ defined by $V(f)(x)=\int_0^x f(t)dt \mbox{ for all }
f\in C[0, 1].$ Then $\sigma_p^0(A)=\sigma_p^0(V)=\emptyset,$ but $A$
and $V$ are not a-isoloid.

 \begin{thm}\label{T4}Suppose that  $S\in L(X)$ and $T\in L(Y)$
property $(z)$ and are a-isoloid, then
 $S\oplus T$ satisfies property $(z)$ if and only if
 $\sigma_{uw}(S\oplus T)=\sigma_{uw}(S)\cup
\sigma_{uw}(T).
$\end{thm}

\begin{proof} ($\Rightarrow$) See the proof of necessity condition of the precedent theorem.

($\Leftarrow$) Since $S$ and $T$ are a-isoloid, then
$$\pi_{00}^a(S\oplus
T)=[\pi_{00}^a(S)\cap\rho_a(T)]\cup[\pi_{00}^a(T)\cap\rho_a(S)]\cup
[\pi_{00}^a(S)\cap \pi_{00}^a(T)],\,\,\mbox{see Lemma \ref{cl}}.$$
On the other hand, as $S$ and $T$ satisfy property $(z)$ then
$\rho(S)=\rho_a(S)$ and $\rho(T)=\rho_a(T).$ Therefore
\begin{eqnarray*} \sigma(S\oplus
T)\setminus\sigma_{uw}(S\oplus
T)&=&[\sigma(S)\cup\sigma(T)]\setminus[\sigma_{uw}(S)\cup\sigma_{uw}(T)]\\
&=&[(\sigma(S)\setminus\sigma_{uw}(S))\cap\rho(T)]\cup[(\sigma(T)\setminus\sigma_{uw}(T))\cap\rho(S)]\nonumber\\
&&\cup[(\sigma(S)\setminus\sigma_{uw}(S))\cap (\sigma(T)\setminus\sigma_{uw}(T))]\nonumber\\
&=& [\pi_{00}^a(S)\cap\rho_a(T)]\cup[\pi_{00}^a(T)\cap\rho_a(S)]\cup[\pi_{00}^a(S)\cap
\pi_{00}^a(T)].\label{eq7}\end{eqnarray*}
Hence $\pi_{00}^a(S\oplus T)=\sigma(S\oplus
T)\setminus\sigma_{uw}(S\oplus T)$ and $S\oplus T$ satisfies property $(z).$
\end{proof}

\noindent To give the reader a good overview of the subject, we
present here another  proof of the sufficient condition of Theorem
\ref{T4}:
\begin{proof}
Since $S$ and $T$ are a-isoloid, then $$\pi_{00}^a(S\oplus
T)=[\pi_{00}^a(S)\cap\rho_a(T)]\cup[\pi_{00}^a(T)\cap\rho_a(S)]\cup
[\pi_{00}^a(S)\cap \pi_{00}^a(T)],$$ and since  $S$ and $T$ satisfy
$(z)$ then from \cite[Theorem 3.6]{Z1}, $\pi_{00}^a(T)=p_{00}^a(T)$
and $\pi_{00}^a(S)=p_{00}^a(S).$ So
 \begin{eqnarray*}\pi_{00}^a(S\oplus
T)&=&[p_{00}^a(S)\cap\rho_a(T)]\cup[p_{00}^a(T)\cap\rho_a(S)]\cup
[p_{00}^a(S)\cap p_{00}^a(T)]\\
&=&[\sigma_a(S)\cup \sigma_a(T)]\setminus[\sigma_{ub}(S)\cup \sigma_{ub}(T)]\\
&=& p_{00}^a(S\oplus T).\end{eqnarray*}
On the other
hand, $S$ and $T$ satisfy also property $(az)$ and since by
hypothesis  $\sigma_{uw}(S\oplus T)=\sigma_{uw}(S)\cup
\sigma_{uw}(T)$  then from Theorem \ref{T1}, $S\oplus T$ satisfies
property $(az).$ Hence $S\oplus T$ satisfies property $(z).$
\end{proof}

\begin{ex}\label{essential} The hypothesis ``$\sigma_{p}^0(S)=\sigma_{p}^0(T)$" in Theorem
\ref{T3} and the hypothesis ``$S$ and $T$ are a-isoloid" in Theorem
\ref{T4} are essential. Indeed,  let $S$ be the operator defined on
$\ell^2(\mathbb{\N})$ by $S(x_1,x_2,x_3,\ldots)=(0,\frac{x_1}{2},\frac{x_2}{3},
\frac{x_3}{4},\ldots)$ and $T$ be a non trivial nilpotent operator
on $\C^n.$  $S$ and $T$ satisfy property $(z)$ since
$\sigma(S)\setminus \sigma_{uw}(S)= \emptyset =\pi_{00}^a(S),$ and
$\sigma(T)\setminus \sigma_{uw}(T)= \{0\}=\pi_{00}^a(T).$ But
$S\oplus T$ does not satisfy property $(z);$  $\sigma(S\oplus
T)\setminus \sigma_{uw}(S\oplus T)= \emptyset$ and
$\pi_{00}^a(S\oplus T)=\{0\}.$ Note that here
$\sigma_{p}^0(S)=\emptyset,$ $\sigma_{p}^0(T)=\{0\}$ and $T$ is
a-isoloid but $S$ isn't.

\end{ex}

 The following theorem gives similar results to   Theorem
\ref{T3} and Theorem \ref{T4}  for  property $(gz).$ The proofs go
similarly.

 \begin{thm}\label{T5} Suppose  $S\in L(X)$ and $T\in L(Y)$ satisfy property $(gz).$\\
i)  If $\sigma_{p}(S)=\sigma_{p}(T),$  or $S$ and $T$ are a-isoloid,
then
 $S\oplus T$ satisfies property $(gz)$ if and only if  $\sigma_{ubw}(S\oplus T)=\sigma_{ubw}(S)\cup
\sigma_{ubw}(T).$
 \end{thm}

\begin{rema} Note that the hypothesis ``$\sigma_{p}(S)=\sigma_{p}(T)$ or $S$ and $T$ a-isoloid" assumed in Theorem
\ref{T5} is essential. For example, the operators $S$ and $T$
defined in Remark \ref{essential}  satisfy property $(gz),$ since
$\sigma(S)\setminus\sigma_{ubw}(S)=\emptyset=\pi_{0}^a(S)$ and
$\sigma(T)\setminus \sigma_{ubw}(T)= \{0\}=\pi_{00}^a(T).$  But,
since $S\oplus T$ does not satisfy property $(z),$ then it does not
satisfy property $(gz)$ too. We note also that the conditions  ``$S$
and $T$ a-isoloid"  and ``$\sigma_{p}(S)=\sigma_{p}(T)$" are
independent as seen in the case of Theorem \ref{T4}.
\end{rema}

\section{Applications}

We begin by recalling   the definition  of the  class of $(H)$-operators,  and definitions of some classes of operators which are
contained in the class $(H).$\\
 According to  \cite{Aie}, the \textit{quasinilpotent} part $H_0(T)$ of $T\in L(X)$ is defined
 as the set
$H_{0}(T)=\{x\in X:
\displaystyle\lim_{n\rightarrow\infty}\|T^{n}(x)\|^{\frac{1}{n}}=0\}.$
Note that   generally, $H_0(T)$  is not closed  and from
\cite[Theorem 2.31]{Aie}, if $H_0(T-\lambda I)$ is closed then $T$
has SVEP at $\lambda.$ We also recall that  $T$ is said to belong to
the class  $(H)$ if for all  $\lambda\in\mathbb{C}$
 there exists
$p:=p(\lambda)\in\mathbb{N}$ such that $H_{0}(T-\lambda
I)=N((T-\lambda I)^p),$ see \cite{Aie} for more details about
this  class of operators. Of course, every operator $T$ which
belongs  the class $(H)$ has SVEP, since $H_0(T-\lambda I)$ is
closed, observe also that $a(T-\lambda I)\leq p(\lambda),$ for every
$\lambda\in\mathbb{C}.$ The class of operators having the property
$(H)$ is large. Obviously, it contains every operator having the
property $(H_1)$. Recall that an operator $T\in L(X)$ is said to
have the property $(H_1)$ if $H_{0}(T-\lambda I)=N(T-\lambda I)$
for all $\lambda\in\mathbb{C}$. Although the property $(H_1)$ seems
to be  strong, the class of operators having the property $(H_1)$ is
considerably large.  Every \textit{totally
paranormal} operator  has property $(H_1)$, and in particular every
hyponormal operator has  property $(H_1)$. Also every
\textit{transaloid} operator or \textit{log-hyponormal}  has the
property $(H_1).$ Multipliers on a semi-simple Banach  algebra
belong to the class $(H_1).$
 Some other operators  satisfy property $(H)$; for example
\textit{M-hyponormal} operators,  \textit{p-hyponormal} operators,
\textit{algebraically p-hyponormal}  operators,
\textit{algebraically M-hyponormal} operators, \textit{subscalar}
operators and
 \textit{generalized scalar} operators. For more details  about
the definitions and comments  about these classes of operators,
  we refer the reader to \cite{Aie}, \cite{CH}, \cite{LN}.

  Now, we give an example of an operator of the class $(H)$
  which does not satisfy the properties $(az)$ and $(z).$
  \begin{ex}\label{ex4}
  Let $T$ be the hyponormal operator given by the direct sum of
  the null operator on $\ell^2(\N)$ and the unilateral right shift $R$ on $\ell^2(\N).$ Then
  $\sigma(T)= D(0, 1);\,\,\, \sigma_a(T)= C(0, 1)\cup \{0\};\,\,\, \sigma_{uw}(T)= C(0, 1)\cup \{0\}$ and
  $\pi_{00}^a(T)=p_{00}^a(T)=\emptyset.$ It follows that $T$ does not satisfy the properties
  $(z)$ and
  $(az).$
  \end{ex}

  In the following proposition we establish the stability of
  properties $(az)$ and $(z)$ by the direct sum of two
  $(H)$-operators.

  \begin{prop}If $S\in L(X)$ and $T\in L(Y)$ are $(H)$-operators satisfying property $(az)$ (resp., property $(z)$) then $S\oplus T$ satisfies property $(az)$ (resp., property $(z)$).
  \end{prop}

  \begin{proof} Since $S$ and $T$ are $(H)$-operators, then they have SVEP and so have a shared stable sign index. From Lemma \ref{lem1}, we have
   $\sigma_{uw}(S\oplus T)=\sigma_{uw}(S)\cup\sigma_{uw}(T).$ Thus, if $S$ and $T$ satisfy $(az)$ then from Theorem \ref{T1}, $S\oplus T$ satisfies property $(az).$
  If $S$ and $T$ satisfy property $(z),$ then  $\sigma(S)=\sigma_a(S)$ and $\sigma(T)=\sigma_a(T).$ This implies (since every $(H)$-operator is isoloid) that $S$ and $T$ are a-isoloid. Then
  we conclude by Theorem \ref{T4} that $S\oplus T$ satisfies property
  $(z).$
  \end{proof}

In the next proposition, we give a similar result for the class of
paranormal operators on Hilbert spaces. We notice that a paranormal
operator may not be in the class of
 $(H)$-operators, for instance see \cite[Example 2.3]{Aie2}. Recall
 that a bounded linear operator $T$ on a Hilbert space $\h$ is said to be paranormal if
$||Tx||^2\leq ||T^2x||\,||x||,$ for all $x\in\h.$

  \begin{prop} If $S\in L(\h)$ and $T\in L(\h)$ are paranormal operators satisfying property  $(az)$ (resp., property $(z)$) then $S\oplus T$
  satisfies property $(az)$ (resp., property $(z)$).
  \end{prop}
\begin{proof} According to \cite{Aie2},
every paranormal operator has the SVEP. Moreover, paranormal operators
are isoloid, see \cite[Lemma 2.3]{CH2}. We conclude as seen in the
proof of last proposition.
\end{proof}

A bounded linear operator $A\in L(X,Y)$ is said to be
\textit{quasi-invertible} if it is injective and has dense range.
Two bounded linear operators $T\in L(X)$ and $S\in L(Y)$ on complex
Banach spaces $X$ and $Y$ are \textit{quasisimilar} provided  there
exist quasi-invertible operators $A\in L(X,Y)$ and $B\in L(Y,X)$
such that $AT=SA$ and $BS=TB$.

  \begin{prop}If $S\in L(X)$ and $T\in L(Y)$ are quasisimilar operators satisfying  property $(z)$ and one of them has the  SVEP, then $S\oplus T$ satisfies property $(z).$
  \end{prop}

  \begin{proof} Quasisimilarity implies the SVEP for both operators,
  and it implies that $\sigma_p^0(S)=\sigma_p^0(T).$ We conclude from
Theorem \ref{T3}.
\end{proof}

\begin{rema} It is well known that if $S\in L(X)$ and $T\in L(Y)$ have the SVEP,  then from \cite[Theorem 2.9]{Aie} the direct sum $S\oplus T$ has the SVEP. This implies that
$\sigma_{ubw}(S\oplus T)=\sigma_{ubw}(S)\cup \sigma_{ubw}(T).$  From
Theorem \ref{T5}, we obtain  analogous  preservation results
established  in the three last propositions for property $(gz).$
\end{rema}
\section{An a-Browder type theorem proof and counterexamples}

In this section we will give a correct proof of \cite[Theorem
2.3]{SZ}. In the original proof in \cite{SZ}, the equality
$$\sigma_p^0(S\oplus T)= \sigma_p^0(S)\cup \sigma_p^0(T)$$ was used
and consequently gave the equality:
$$\pi_{00}^a(S\oplus T)=\mbox{iso}[\sigma_a(S)\cup \sigma_a(T )]
\cap [\sigma_p^0(S)\cup \sigma_p^0(T)],$$ (see line 8 of the proof
of \cite[Theorem 2.3]{SZ}.)

 But these last 2 equalities
are false, see examples \ref{ex5a} and \ref{ex5b}.

We recall   that  an operator $T\in L(X)$ satisfies property $(sbaw)$
if $\sigma_a(T)\setminus\sigma_{ubw}(T)=\pi_{00}^a(T).$ In the
following theorem, we give the same version   of \cite[Theorem 2.3]{SZ} followed by  a correct
proof.

\begin{thm}
 Let $S\in L(X)$ and $T\in L(Y ).$ If S and T have property
(sbaw) and are a-isoloid, then the following assertions are
equivalent: \\
(i) $S\oplus T$ has property $(sbaw);$\\
 (ii) $\sigma_{ubw}(S\oplus T ) =
\sigma_{ubw}(S)\cup \sigma_{ubw}(T).$
\end{thm}
\begin{proof} (i) $\Longrightarrow$ (ii)  The property $(sbaw)$ for
$S\oplus T$ implies (ii) with no other restriction, since form
\cite{KZZ}, $S\oplus T$ satisfies generalized a-Browder's theorem,
and hence  by \cite[Lemma 2.1]{SZ}, $\sigma_{ubw}(S\oplus
T)=\sigma_{ubw}(S)\cup \sigma_{ubw}(T)$. \\
(ii) $\Longrightarrow$
(i) Suppose that $\sigma_{ubw}(S\oplus T)=\sigma_{ubw}(S)\cup
\sigma_{ubw}(T).$ Since $S$ and  $T$ are a-isoloid   then
\begin{eqnarray} \pi_{00}^a(S\oplus T)
&=&[\pi_{00}^a(S)\cap\rho_a(T)]\cup[\pi_{00}^a(T)\cap\rho_a(S)]\cup
[\pi_{00}^a(S)\cap \pi_{00}^a(T)], \mbox{see Lemma
\ref{cl}}.\nonumber
\end{eqnarray}
 As $S$ and $T$
satisfy property $(sbaw),$ then
\begin{eqnarray}\sigma_a(S\oplus T)\setminus
\sigma_{ubw}(S\oplus T)&=& [\sigma_a(S)\cup\sigma_a(T)]\setminus
[\sigma_{ubw}(S)\cup\sigma_{ubw}(T)]\nonumber\\
&=&
[\pi_{00}^a(S)\setminus\sigma_a(T)]\cup[\pi_{00}^a(T)\setminus\sigma_a(S)]\cup
[\pi_{00}^a(S)\cap \pi_{00}^a(T)].\nonumber
\end{eqnarray}
Hence  $\sigma_a(S\oplus T)\setminus\sigma_{ubw}(S\oplus T)=
\pi_{00}^a(S\oplus T)$ and so  property $(sbaw)$ is satisfied by
$S\oplus T$.
\end{proof}

We recall that $\sigma_p(S\oplus T)= \sigma_p(S)\cup \sigma_p(T)$ is
always true. On the other hand we have always $\sigma_p^0(S\oplus
T)\subset \sigma_p^0(S)\cup \sigma_p^0(T),$ but this inclusion may
be proper as we can see in the following examples.
\begin{ex}\label{ex5a}
Let $P\in L(\ell^2(\mathbb{\N}))$ be defined  by
 $P(x_1,x_2,\ldots)=(0,x_2,x_3, x_4,\ldots).$ Take $S=P$ and $T=I-P.$ Then
 $\sigma_p^0(S)=\{0\}$ and $\sigma_p^0(T)=\{1\}$ but $\sigma_p^0(S\oplus
T)=\emptyset.$ Note that $S$ and $T$ are a-isoloid.
\end{ex}
\begin{ex} \label{ex5b} Let $R$ be the unilateral right shift on $\ell^2(\mathbb{\N}).$ We
define $S=R\oplus P$ and $T= U\oplus 0$ where $P$ is the projection
defined in Example \ref{ex5a} and $U$ is defined as follows (see \cite{GK}):\\ $U:
\ell^1(\mathbb{\N})\rightarrow \ell^1(\mathbb{\N})$;    $$U(x) = (0,
a_1x_1, a_2x_2, . . . , a_kx_k, . . .)\,\,\,\forall\,\,x=(x_i)\in
\ell^1(\mathbb{\N})
$$
 where $(a_i)$ is a sequence of complex numbers such that $0 < |a_i |
\leq 1$  and $\sum_{i=1}^{\infty} |a_i | < \infty.$\\
 Then $S\in L(X)$ and $T\in L(Y)$ where $X$ and $Y$ are the Banach
 spaces $\ell^2(\mathbb{\N})\oplus\ell^2(\mathbb{\N})$ and
 $\ell^1(\mathbb{\N})\oplus\ell^1(\mathbb{\N})$ respectively. And we
 have:

 $$\sigma_a(S)=C(0, 1)\cup \{0\},\,\,\,\sigma_{ubw}(S)=C(0, 1) \,\,\mbox{and}\,\,
 \sigma_p^0(S)= \{0\}= \pi_{00}^a(S).$$
 $$\sigma_a(T)=\{0\}= \sigma_{ubw}(T)\,\,\mbox{and}\,\,
\sigma_p^0(T)= \emptyset= \pi_{00}^a(T).$$ It follows that both $S$
 and $T$ satisfy property $(sbaw)$ and $\sigma_p^0(S\oplus T)\neq \sigma_p^0(S)\cup
 \sigma_p^0(T),$ since  $\sigma_p^0(S\oplus T)= \emptyset$ and $\sigma_p^0(S)\cup
 \sigma_p^0(T)= \{0\}.$
 Note that $S$ and $T$ are also a-isoloid and
 $$\emptyset=\pi_{00}^a(S\oplus T)\neq\mbox{iso}[\sigma_a(S)\cup \sigma_a(T )]
\cap [\sigma_p^0(S)\cup \sigma_p^0(T)]=\{0\}.$$
 \end{ex}

\goodbreak
{\small \noindent Abdelmajid Arroud,\\
\noindent Laboratory (L.A.N.O),\\
\noindent Department of Mathematics,\\
\noindent Faculty of Science, Mohammed I University ,\\
\noindent PO Box 717, Oujda 60000 Morocco.\\
\noindent arroud.bm@gmail.com\\

\noindent \noindent Hassan  Zariouh,\newline D\'epartement de Math\'ematiques,\newline Centre R\'egional pour les M\'etiers
de l'\'Education\newline et de la Formation de  l'Oriental (CRMEFO),\newline
  Oujda 60000 Morocco.\newline
 Laboratoire (L.A.N.O), Facult\'e des Sciences,\newline  Universit\'e Mohammed I.\newline
 \noindent h.zariouh@yahoo.fr

\end{document}